\documentclass{birkjour}
\usepackage{enumerate}
\usepackage{mathrsfs}
\usepackage[numbers,sort&compress]{natbib}
 \newtheorem{thm}{Theorem}[section]
 \newtheorem{cor}[thm]{Corollary}
 \newtheorem{lem}[thm]{Lemma}
 \newtheorem{prop}[thm]{Proposition}
 \theoremstyle{definition}
 
 \theoremstyle{remark}
 \newtheorem{rem}[thm]{Remark}
 \newtheorem*{ex}{Example}
 \numberwithin{equation}{section}
 \theoremstyle{assumption}
 \newtheorem{assumption}[thm]{Assumption}

 \newcommand{\eps}{\varepsilon}
 \newcommand{\norm}[1]{\Vert#1\Vert}
 
 \newcommand{\abs}[1]{\left\vert#1\right\vert}
 \newcommand{\set}[1]{\left\{#1\right\}}
 
 \newcommand{\inner}[1]{\left(#1\right)}

\begin{document}

%
%
%
%
%
%
%
%
%

\title[Compactness of  the resolvent for the Witten Laplacian]
 {Compactness of  the resolvent for the Witten Laplacian}

\author[W.-X. Li]{Wei-Xi Li}

\address{%
School of Mathematics and Statistics,  and Computational Science Hubei Key Laboratory, Wuhan University,
Wuhan 430072, China }

\email{wei-xi.li@whu.edu.cn}

\thanks{The work was supported by NSF of China( No. 11422106) and Fok Ying Tung Education Foundation (No. 151001).}


\date{}

\begin{abstract}
In this paper  we consider  the  Witten Laplacian on 0-forms and give sufficient conditions  under which the  Witten Laplacian admits a compact resolvent.  These conditions  are imposed on the potential itself,   involving the control of high order derivatives by lower ones, as well as the control of the positive eigenvalues of the Hessian matrix.  This compactness  criterion for  resolvent  is  inspired by the one  for  the Fokker-Planck operator.  Our method relies on   the   nilpotent group techniques  developed by Helffer-Nourrigat [Hypoellipticit\'e maximale pour des
 op\'erateurs polyn\^omes de champs de vecteurs, 1985].
\end{abstract}

\maketitle
\section{Introduction and main results}
The Witten Laplacians on forms were initially introduced by E. Witten \cite{Witten} on a compact manifold,  where he considered  a  new complex  associated with the  distorted exterior
\begin{eqnarray*}
	d_{V}=e^{-V}\circ d \circ e^{V}.
\end{eqnarray*}
Then the Witten Laplacians on forms are  defined by 
\begin{eqnarray*}
	\Delta_{V}^{(\cdot)}= d_{V}\circ d_{V}^*+ d_{V}^*\circ d_{V}.
\end{eqnarray*}
 In this paper  we will consider only    the   Witten Laplacians on 0-forms  in the real space  $\mathbb R^n$, and in this case it reads  
\begin{eqnarray*}
   \Delta_{V}^{(0)}=
  -\Delta_x+\abs{\partial_xV(x)}^2-\Delta_x
  V(x),
 \end{eqnarray*}
 and  has the form of   a    Schr\"odinger operator $-\Delta+\tilde V$ with $\tilde V=\abs{\partial_xV(x)}^2-\Delta_x
  V(x)$.  If  replacing  $V$ and $d$ respectively  by $V/h$ and $hd$ in the distorted exterior  we then get  the   semi-classical Witten Laplacian  $$\triangle_{ V,h}^{(0)} =-h^2\triangle_x + \abs{\partial_x V}^2-h \triangle_x   V.$$   
  It is of interest in itself to analyze   the spectrum of   the semi-classical Witten Laplacian as  the parameter  $h\rightarrow 0,$  cf. \cite{He02,HeNi06, HeNi04, Le, LeNiVe, Ma, Mi}  and the references listed therein.  If we introduce another parameter by 
  \begin{eqnarray*}
  	\tau=h^{-1},
  \end{eqnarray*} 
  then the semi-classical Witten Laplacian can be rewritten as
  \begin{eqnarray*}
  h^{-2}	\triangle_{ V,h}^{(0)} =-\triangle_x +\tau^{2} \abs{\partial_x V}^2-\tau \triangle_x   V=\triangle_{ \tau V}^{(0)}.
  \end{eqnarray*}
  The latter operator  is also closely related to the  microhypoellipticity  problem for   the  system  of complex vector fields   
  \begin{equation}\label{syseqcomplex}
  P_j=\partial_{x_j}- i  \inner{\partial_{x_j} V (x)}\partial_{t},\quad
  j=1,\cdots,n,\quad i=\sqrt{-1},
\end{equation}
where limit $\tau\rightarrow+\infty$ has to be considered.

Our main goal of this paper is to explore  the  criterion by  which the Witten Laplacian   has a compact resolvent and thus admits purely discrete spectrum.  This issue is closely linked with the exponential trend to the equilibrium for the  spatially inhomogeneous kinetic systems, such as the   non-selfadjoint   Fokker-Planck and Boltzmann equations, cf.\cite{DeVi, Herau15,Herau07,Herau06,HerauNier04}.
 Similar problems occur in the theory of the $\bar\partial$-Neumann problem, and we refer to \cite{ChrFu,FuSt,FuSt1,HaHe} and the surveys  given in \cite{Ha,St}, which reveal  that there is a close relationship between the  Witten Laplacians  and the weighted $\Box_b$-operator of the  $\bar\partial$-complex.

 By one of the elementary results on Schr\"odinger operators we see  the Witten Laplacian is with a compact resolvent  if 
 \begin{eqnarray*}
 \abs{\partial_xV(x)}^2- \Delta_x V\rightarrow+\infty,   ~~{\rm as} ~~\abs x\rightarrow+\infty.
\end{eqnarray*} 
More generally (see \cite{He02,HelfferNier05} for instance),  it is still true if
 \begin{eqnarray*} 
	t\abs{\partial_xV(x)}^2-\Delta_x V\rightarrow+\infty,   ~~{\rm as} ~~\abs x\rightarrow+\infty
\end{eqnarray*}
for some $t\in]0,2[.$  The subject  of compact resolvent for Witten Laplacian has already been explored extensively  by  Helffer-Nier \cite{HelfferNier05}   based on the idea of nilpotent Lie groups.    This  idea  was initiated by Rothschild-Stein \cite{RS} when studying the hypoellipticity property of the H\"ormander's operators   and    Rothschild-Stein lifting theorem says that one can obtain the sharp local regularity  by lifting  the vector fields  to nilpotent Lie groups and then using  the  analysis for the corresponding left invariant operators defined on the groups.   
This kind of nilpotent Lie techniques were 
     developed further  by Nourrigat \cite{Nourrigat2,Nourrigat1,Nourrigat3} and  Helffer-Nourrigat \cite{Hel-Nou}  for systems of pseudo-differential operators, where the pseudo-differential operators are approximated by operators defined in Euclidean space with polynomial coefficients and      the problem is then reduced to the analysis of the operators with polynomial coefficients.  When applying the nilpotent techniques to study the maximal estimate for the  specific Witten Laplacian,  the property can be deduced from the analysis of the ``limiting polynomials'' (see \cite{HelfferNier05}  or  Subsection \ref{sub-hn} below for the precise definition), and based on this idea  Helffer-Nier  \cite{HelfferNier05}  have obtained  the compact criteria for the resolvent of Witten Laplacian with  specific  potentials  such as  polynomials,  the  homogeneous and polyhomogeneous  functions and the analytic functions.     We  also  refer to the work of Helffer-Mohamed \cite{HM} for the first application of  the hypoellipticity  techniques to the compactness problems in mathematical physics.      In this work we will give a new criterion,   involving the similar conditions related to  the local minimum problem. We remark that these conditions are  imposed on the potential $V$ itself rather than on the  ``limiting polynomials'',  which is somehow easy to check and apply to concrete examples.  The criterion is inspirited by the one for Fokker-Planck operator \cite{lipreprint} and the Helffer-Nier's conjecture.  The Helffer-Nier's conjecture says the Fokker-Planck operator 
   has a compact resolvent if and
only if the Witten Laplacian $\Delta_{V/2}^{(0)}$ 
 has a compact  resolvent, which shows  the close
 link between the compact resolvent property for the Fokker-Planck operator and the same property for the corresponding Witten Laplacian. 
 The necessity part, that is the Witten
Laplacian $\Delta_{V/2}^{(0)}$ has a compact resolvent  if the
Fokker-Planck operator $P$ is with a compact resolvent,  has already proven by Helffer and Nier (c.f. \cite[Theorem 1.1]{HelfferNier05}). The reverse implication   still remains open up to now, except  for some special   potentials (cf. \cite{HelfferNier05,HerauNier04,lipisa} ).  Recently  the author \cite{lipreprint}  obtained a  compactness criteria  for  the resolvent of Fokker-Planck operator,  involving  the control of the positive eigenvalues of  the Hessian matrix of the  potential, 
 and the main assumption on $V$ there is that
     \begin{eqnarray}\label{assext}
    \forall~x\in\mathbb R^n,\quad	\sum_{j\in I_{x}} \lambda_j(x)\leq  C  \inner{1+\abs{\partial_x
  V(x)}^2}^{2/3},
    \end{eqnarray} 
    with 	  $\lambda_\ell, 1\leq \ell\leq n$,  the eigenvalues of the Hessian  matrix   
  $
 \inner{\partial_{x_ix_j}V}_{1\leq i, j\leq n} 
$ and  
    \begin{eqnarray}\label{indexset}
  	I_{ x}= \Big\{1\leq \ell \leq n;~\lambda_\ell(x)>0\Big\}. 
  \end{eqnarray}
Under the assumption \eqref{assext}  the author  \cite{lipreprint} proved  the Fokker-Planck operator admits  a compact resolvent, provided for some   $\alpha\geq 0,$ 
  \begin{equation}\label{co2}
	\lim_{\abs x\rightarrow +\infty} \abs{\partial_x V(x)}=+\infty, \quad \textrm{or}~~ \lim_{\abs x\rightarrow +\infty} \Big(\alpha\abs{\partial_x V(x)}^2-\Delta_x V(x)\Big)=+\infty.
\end{equation}
In view of the   necessity part of the Helffer-Nier's Conjecture  \cite[Theorem 1.1]{HelfferNier05}, we see Witten Laplacian   has also a compact resolvent under the same assumptions \eqref{assext} and \eqref{co2}.  As far as the Witten Laplacian is only concerned, we can go further by improving the conditions  \eqref{assext} and \eqref{co2}.    
Precisely, the  main assumption on $V$  can be stated as follows.  

  Throughout the paper we will let $k\geq 2$ be a given integer, and define $\tilde f$ by setting  \begin{equation}\label{defoffv}
\tilde f(x)\stackrel{\rm def}{=} \sum_{1\leq \abs\alpha \leq k} \abs{\partial^\alpha  V(x)}^{1/\abs\alpha}.	
\end{equation}

\begin{assumption}\label{assu}
Letting $I_x$ and $\tilde f$ are  given respectively   in \eqref{indexset} and \eqref{defoffv}, we  suppose the assumptions listed subsequently are fulfilled. 
\begin{enumerate}[(i)]
	\item A constant $C>1$ exists such that for any $x\in\mathbb R^n$ we have 
  \begin{equation}\label{maiass}
	 	\sum_{j\in I_{x}} \lambda_j(x)\leq  C  \Big(\mathcal M(x)+ \abs{\partial_x
  V(x)}^2+\sum_{2\leq \abs\alpha \leq k}\abs{\partial^\alpha  V(x)}^{\inner{2-\delta_1}/\abs\alpha} +1\Big),
\end{equation}
where $0<\delta_1<1$ is an arbitrarily small   number, and $\mathcal M$ is defined by 
\begin{eqnarray*}
	\mathcal M (x)=\sum_{j\not\in I_x}\big(-\lambda_j(x)\big).
\end{eqnarray*}
 Recall  $\lambda_\ell$  are  the eigenvalues of the Hessian  matrix   
  $
 \inner{\partial_{x_ix_j}V}_{1\leq i, j\leq n}.
$
\item There exists an arbitrarily small number $0<\delta_2<1,$   such that  \begin{equation}\label{conupp}
 \forall~\abs\alpha= k+1,~\forall~x\in\mathbb R^n,\quad  \abs{\partial^\alpha V(x)}	\leq C_\alpha \Big(1+\tilde f(x)\Big)^{k+1-\delta_2},
 \end{equation}
 where  $C_\alpha$ are constants depending only on $\alpha$. 
 \item We have  $\tilde f(x)  \rightarrow +\infty$ as $\abs x\rightarrow+\infty.$
\end{enumerate}
  \end{assumption}

 The main results can be stated as follows.

\begin{thm}\label{thmmain}
	Under Assumption \ref{assu},      we can find two positive constant $C,\tau_0>0,$   such that
		\begin{eqnarray}\label{cesc}
		\forall~\tau\geq\tau_0,~~\forall~u\in C_0^\infty(\mathbb R^n),\quad  \norm{ \tilde f_\tau u}_{L^2}^2
     \leq C   \inner{\triangle_{\tau V}^{(0)} u, u}_{L^2} + C\norm{u}_{L^2}^2,
	\end{eqnarray}  
	where $\tilde f_\tau$ is defined by
	\begin{eqnarray*}
		\tilde f_\tau(x)=\sum_{1\leq \abs\alpha \leq k} \tau^{1/\alpha}\abs{ \partial^\alpha  V(x)}^{1/\abs\alpha}.
	\end{eqnarray*} 
	Moreover for any $\tau$ with $0<\tau<\tau_0$ we can find a constant $C_\tau,$ depending on $\tau$ and the constant $C$ in \eqref{cesc}, such that
	\begin{equation}\label{lastestimate}
		~\forall~u\in C_0^\infty(\mathbb R^n),\quad  \norm{ \tilde f_\tau u}_{L^2}^2
     \leq C_\tau   \inner{\triangle_{\tau V}^{(0)} u, u}_{L^2} + C\norm{u}_{L^2}^2.
	\end{equation}
As a result the Witten Laplacian $\triangle_{\tau V}^{(0)}$  has a  compact resolvent for any $\tau>0.$ 
\end{thm}

\begin{rem}
\begin{enumerate}[\noindent(i)]
	\item   In view of  Helffer-Nier's conjecture we may expect that the Fokker-Planck operator is also with a compact resolvent under   Assumption \ref{assu} above.   
\item We need only verify  the estimate   \eqref{maiass} for these points  where $\Delta V$ is positive,
	since  it obviously holds for the points where $\Delta V\leq 0.$
\item  As an immediate consequence of  \eqref{cesc} we have the   maximal microhypoellipticity in the  direction $\tau>0$  (see \cite{HelfferNier05}) for the system \eqref{syseqcomplex}, that is,  a constant $C$ exists such that for any $\tau>0$ and for any $u\in C_0^\infty(\Omega)$  we have   	\begin{eqnarray*}
		  \norm{  \partial_{x} u}_{L^2}^2+\tau^2\norm{  \inner{\partial_{x} V} u}_{L^2}^2 
     \leq C  \norm{P_{\tau} u}_{L^2 }^2 + C\norm{u}_{L^2}^2,
	\end{eqnarray*}  
where $\Omega$ is a neighborhood of the point  $x_0$ such that $\partial_x V(x_0)=0,$ and    	 $$P_{\tau}=\partial_{x}+ \inner{\partial_{x} V}  \tau.$$ 
Note that 
$\norm{P_{\tau} u}_{L^2 }^2=\big(\triangle_{ \tau V}^{(0)} u, u\big)_{L^2},$ and the constant $C$ here may depend on $\Omega.$ 
 \end{enumerate}
\end{rem}

 We end   the introduction with two examples which 
  has already been explored in \cite{HelfferNier05}  by virtue of the "limiting polynomials" (see Subsection \ref{sub-hn} below for the precise definition). 
 
 \begin{ex}[Example 10.4.3.1 of \cite{HelfferNier05} ]
 	Consider the potentials  $$ V_\delta(x)= x_1^2x_2^2+\delta(x_1^2+x_2^2),$$ with $\delta\neq 0$ given real number.
 \end{ex}

 Here we apply Theorem \ref{thmmain}.  It  is clear the statements (ii)-(iii) with $k=4$ in Assumption \ref{assu} are fulfilled for the above $V_\delta.$  It remains to check the estimate   \eqref{maiass} in the first statement (i).  Direct computation gives $$\abs{\nabla V_\delta(x)}^2=4x_1^2\inner{x_2^2+\delta}^2+4x_2^2\inner{x_1^2+\delta}^2$$ and  $$\sum_{j\in I_x}\lambda_j(x)+\sum_{j\notin I_x}\lambda_j(x)=\Delta V_\delta(x)=2  x_1^2+2  x_2^2+4\delta.$$

 If $\delta>0$, then  for any $x\in\mathbb R^n,$
\begin{eqnarray*}
	\sum_{j\in I_x}\lambda_j(x)+\sum_{j\notin I_x}\lambda_j(x) \leq 2\delta^{-2} \abs{\partial_x V_\delta(x)}^2+4\delta.
\end{eqnarray*}
Then the property in \eqref{maiass} is satisfied   for   $\delta>0.
$

Now suppose $\delta<0.$ For each $(x_1.x_2)\in\mathbb R^2$ we have three cases.
\begin{enumerate}[(a)]
	\item  If $\abs{x_1^2+\delta}\leq -\delta/2$ and $\abs{x_2^2+\delta}\leq -\delta/2$ then $$\sum_{j\in I_x}\lambda_j(x)+\sum_{j\notin I_x}\lambda_j(x)\leq -2\delta,$$ and thus \eqref{maiass} holds. 
	\item If $\abs{x_1^2+\delta}\geq -\delta/2$ and $\abs{x_2^2+\delta}\geq -\delta/2$ then
	\begin{eqnarray*}
		\abs{\nabla V_\delta(x)}^2\geq \delta^2\abs x^2\geq  \frac{\delta^2}{2} \Big(\sum_{j\in I_x}\lambda_j(x)+\sum_{j\notin I_x}\lambda_j(x)-4\delta\Big),
	\end{eqnarray*}
	which yields \eqref{maiass}.
	\item One of the terms $\abs{x_j^2+\delta}, j=1,2,$ is bigger than $-\delta/2$ and the another one is smaller than $-\delta/2$. We may suppose without loss of generality that $\abs{x_1^2+\delta}\leq -\delta/2$ and $\abs{x_2^2+\delta}\geq -\delta/2.$ Then  we have
	\begin{equation}\label{twoca}
		-{\delta \over 2}\leq x_1^2\leq  -\frac{3\delta}{2} ~\textrm{and}~x_2^2\geq -\frac{3\delta}{2}, ~\textrm{or}~-\frac{\delta}{2}\leq x_1^2\leq  -\frac{3\delta}{2}~\textrm{and}~x_2^2\leq -\frac{\delta}{2}.
	\end{equation}
Note that $\Delta V_\delta$ is bounded from above by a constant  for the latter case in \eqref{twoca} and thus \eqref{maiass}  holds; meanwhile for the former case we have, observing $\abs{x_2^2+\delta}=x_2^2+\delta\geq -\delta/2,$
\begin{eqnarray*}
	\abs{\nabla V_\delta(x)}^2\geq (-2\delta)(x_2^2+\delta)^2\geq \delta^2(x_2^2+\delta)  
 \end{eqnarray*}
and
\begin{eqnarray*}
	\sum_{j\in I_x}\lambda_j(x)+\sum_{j\notin I_x}\lambda_j(x)\leq 2x_2^2+\delta,
\end{eqnarray*}
which also yields \eqref{maiass}.
\end{enumerate} 
We then conclude that  $\triangle_{ V_\delta}^{(0)}$ has a compact resolvent 
whenever $\delta\neq 0.
$  Note that  $\triangle_{ V_\delta}^{(0)}$ doesn't have    compact resolvent  when $\delta=0.$

\begin{ex}[Example 10.4.3.2 of \cite{HelfferNier05} ]
Consider the potential $\Phi_\delta$ defined by
\begin{eqnarray*}
	\Phi_{\delta}=(x_1^2-x_2)^2+\delta x_2^2.
\end{eqnarray*}
\end{ex}

By Proposition 10.21 of \cite{HelfferNier05} we see $\Delta_{\Phi_\delta}^{(0)}$ has a compact resolvent if (and only if) $\delta\neq 0.$   If using   
Theorem \ref{thmmain} instead we can conclude  that it is  true for $\delta\in\mathbb R\setminus \big\{0, -1\big\},$  and  so our results can't apply to this example when $\delta=-1.$     
To see this we need only verify the condition \eqref{maiass}.     Direct calculation gives 
\begin{eqnarray*}
	\abs{\partial_x \Phi_\delta} \sim 2\abs{x_1^2-x_2}\cdot \abs{x_1}+\abs{x_1^2-(1+\delta)x_2},
\end{eqnarray*}
and the Hessian matrix $H_{\Phi_\delta}$ of  $\Phi_\delta$ reads
\begin{eqnarray*}\label{diag}
	H_{\Phi_\delta}= \begin{pmatrix}
		&12x_1^2-4x_2  &-4x_1\\
		&-4x_1 &  2(1+\delta)x_2 \\
	\end{pmatrix} ,
\end{eqnarray*}
Write $ \mathbb R^2=  A_1\cup  A_2$ with
\begin{eqnarray*}
	    A_1= \Big\{ x=(x_1,x_2);\  \abs{x_1}\geq 1 \Big\}, \quad  A_2= \Big\{  x=(x_1,x_2); \ \abs{x_1}\leq 1 \Big\}.
\end{eqnarray*}
  
  \medskip
\noindent \underline{\it Consider the case of $x\in  A_1.$} Then
\begin{eqnarray*}
	\abs{\partial_x \Phi_\delta} \sim 2\abs{x_1^2-x_2}\cdot \abs{x_1}+\abs{x_1^2-(1+\delta)x_2}\geq  \abs{x_1^2-x_2} + \abs\delta\cdot \abs{x_2}
\end{eqnarray*}
  Then the modulus of each entry in the matrix $H_{\Phi_\delta}$  is bounded from above by  $\abs{\partial_x \Phi_\delta}+1,$ provided $\delta\neq 0.$ Thus  the condition (1.6) holds. 
  
  \medskip
\noindent \underline{\it Consider the case of $x\in  A_2.$} Then the modulus of each entry in $H_{\Phi_\delta}$  is  bounded by  $\abs{x_2}+1.$ Moreover observe 
\begin{eqnarray*}
	\abs{\partial_x \Phi_\delta} \sim 2\abs{x_1^2-x_2}\cdot \abs{x_1}+\abs{x_1^2-(1+\delta)x_2}\geq   \abs{1+\delta}  \cdot \abs{x_2}-1.
\end{eqnarray*}
 Thus the condition (1.6) holds in this case provided $\delta \neq -1.$
As a result,  it follows from Theorem \ref{thmmain} that  $\triangle_{ \Phi_\delta}^{(0)}$ has a compact resolvent 
whenever $\delta\neq 0, -1.
$   

Finally we remark that  $\Phi_\delta$ violates \eqref{maiass} for $\delta=-1$ at the points $(0,x_2)$ with $x_2\rightarrow -\infty.$   Nonetheless,  $\triangle_{ \Phi_{-1}}^{(0)}$ has indeed  compact resolvent (see \cite [Proposition 10.21]{HelfferNier05}).     
It remains interesting to improve the condition \eqref{maiass} as sharp as possible,  such that it  can be applied to the potential   
$	\Phi_{-1}=x_1^4-2x_1^2x_2.
$

\section{Proof of the main result}
The proof  is strongly inspired by  the  Helffer and Nourrigat's  recursion argument  related to Kirillov's theory, cf.\cite{HelfferNier05, Nier-BJ} for the induction arguments for Witten Laplacian and \cite{Hel-Nou,Nourrigat2,Nourrigat1} for more general problems.  Here we will follow the argument in the Nier's lectures \cite{Nier-BJ}   and  proceed through  the   subsections as below.

\subsection{Helffer and Nourrigat's  Criteria for maximal estimates}\label{sub-hn}

In this part we recall the criteria for the maximal hypoellipticity developed by  Helffer and Nourrigat  \cite{Hel-Nou} and its application to Witten Laplacian(see  \cite{HelfferNier05}).

 Denote by $E_r$ the set of polynomials with degree less than or equal to $r.$  A subset $\mathcal L$ of $E_r$  is called canonical if it has the following properties : 
\begin{enumerate}[\quad (i)]
	\item If $p\in\mathcal L$ and $y\in\mathbb R^n,$  then the polynomial defined by 
	\begin{eqnarray*}
		 q(x)=p(x+y)-p(y)
	\end{eqnarray*}
	also lies in $\mathcal L.$
	\item  If    $p\in\mathcal L$ and $\delta>0,$  then polynomial
	$
		 q(x)=p(\delta x)
	$
	also lies in $\mathcal L.$
	\item $\mathcal L$ is a closed subset of $E_r$.  
\end{enumerate}

Given $p\in E_r$, we denote by $\mathcal L_{p,0}$  the  canonical set which contains  all  the polynomials $q$ of degree less than or equal to $r$  vanishing at $0$ such that    there exists  a sequence   $y_j\in\mathbb R^n$ with $y_j\rightarrow 0$  and   sequences  $\tau_j$ and $h_j$  of positive  numbers  with $\tau_j\rightarrow+\infty $ and $h_j\rightarrow 0,$  such that
   $$\forall ~1\leq \abs\alpha\leq r,\quad \partial^\alpha q(0)=\lim_{j\rightarrow \infty}  \tau_j h_j^{\abs{\alpha}}\partial_x^{\alpha} p(y_j).$$
   
   \begin{rem}\label{reklimit}
   If	 $q\in \mathcal L_{p,0}$ then there exists  a sequence   $y_j\in\mathbb R^n$ with $y_j\rightarrow 0$  and   sequences  $\tau_j$ and $h_j$  of positive  numbers  with $\tau_j\rightarrow+\infty $ and $h_j\rightarrow 0,$  such that
   \begin{eqnarray*}
   	q(x)=\lim_{j\rightarrow+\infty} \tau_j\big[p(y_j+h_j x)-p(y_j)\big].
   \end{eqnarray*}
   \end{rem}

Applying the results of Helffer-Nourrigat \cite{HelfferNier05} to the system \eqref{syseqcomplex} gives the following 
\begin{thm}[Helffer and Nourrigat \cite{HelfferNier05}]\label{prohel}
	Given $p\in E_r.$    Assume that   any nonzero  $q\in\mathcal  L_{p,0}\cap E_{r-1} $  has no local minimum in $\mathbb R^n.$   Then  there exists  a  constant $C>0$ and  a neighborhood $\Omega$ of   $0,$   such that the following estimate 
	\begin{eqnarray*}
   \norm{\partial_{x}u}^2_{L^2}
    +\tau^2\norm{  \inner{\partial_{x}p} u}^2_{L^2}\leq
   C \inner{\Delta_{\tau p}^{(0)}u,~u}_{L^2}+C\norm{u}_{L^2}^2
	\end{eqnarray*}
	holds for all $\tau>0$ and for all $u\in  C_0^\infty\inner{\Omega}.$    
	\end{thm}

\subsection{Stability}
In this part we will show that a stronger form of the estimate \eqref{maiass} is stable for the canonical set introduced above. 
  
We first introduce some notations to be used throughout the paper.  Given a function $\rho\in C^2(\mathbb R^n)$, we denote by $\lambda_{\rho,\ell},1\leq \ell\leq n$ the eigenvalues of the Hessian matrix $\inner{\partial_{x_i}\partial_{x_j}\rho}_{1\leq i, j\leq n}.$  And define $I_{x,\rho}$ and $\mathcal M_\rho(x)$ by setting  
\begin{eqnarray}\label{Ixprho}
I_{x,\rho}=\big\{1\leq \ell\leq n;~\lambda_{\rho,\ell}(x)>0\big\}	
\end{eqnarray}
and
\begin{eqnarray}\label{mrho}
\mathcal M_\rho(x)= \sum_{j\not\in I_{x,\rho}}\big(-\lambda_{\rho,j}(x)\big) =\sum_{j\not\in I_{x,\rho}}\abs{\lambda_{\rho,j}(x)}=\sum_{j=1}^n\max\big\{-\lambda_{\rho,j}(x),0\big\}.
\end{eqnarray}  
We denote  by $B_\sigma$ the ball centered at $0$ with radius $\sigma, $ i.e.,
\begin{eqnarray} \label{bal}
	B_\sigma=\big\{x\in\mathbb R^n;\quad \abs x<\sigma\big\}.
\end{eqnarray}

The main result of this subsection can be stated as follows. 

 \begin{lem}\label{lemcr}
 Let $p\in E_r$  satisfy that  
 \begin{equation}\label{lambdap}
	\forall~x\in B_{\sigma},\quad	\sum_{j\in I_{x,p}} \lambda_{p,j}(x)\leq  C  \inner{\mathcal M_p(x)+ \abs{\partial_x
  p(x)}^2}
\end{equation}
for some $\sigma>0$ and for some constant $C>0,$ where we  use the notations given in \eqref{Ixprho}-\eqref{bal}. 
  Then  there exists a constant $\tilde C\geq 1,$   depending only on the constant $C$ above and the dimension $n,$ such that for any $q\in \mathcal L_{p,0}$ we have 
\begin{equation}\label{eq+1}
	\forall~x\in\mathbb R^n,\quad	\sum_{j\in I_{x,q}} \lambda_{q,j}(x)\leq   \tilde C  \inner{\mathcal M_q(x)+ \abs{\partial_x
  q(x)}^2}.
\end{equation}
 As a result any $q\in\mathcal L_{p,0}\setminus\set{0}$ can not  have any local minimum in  $\mathbb R^n.$ 
 \end{lem}

 \begin{proof} We begin with the proof of the first property  \eqref{eq+1}.    For any $q\in \mathcal L_{p,0}$, by Remark \ref{reklimit} we can find   a sequence   $y_j\in\mathbb R^n$  with $y_j\rightarrow 0$ and sequences $\tau_j$ and $h_j$ of positive  numbers  with $\tau_j\rightarrow+\infty$ and $h_j\rightarrow 0,$ such that
   $$  q(x)=\lim_{j\rightarrow+\infty} \tau_j\big[p(y_j+h_jx)-p(y_j)\big]= \sum_{1\leq\abs\beta\leq r} \inner{\lim_{j\rightarrow+\infty} \tau_j h_j^{\abs{\beta}}\partial^{\beta} p(y_j)}\frac{x^\beta}{\beta!}.$$
   This implies
   \begin{eqnarray}\label{alq}
   \forall~\abs{\alpha}\geq 1,\quad	\partial_x^\alpha q(x)=\sum_{1\leq\abs\beta\leq r} \inner{\lim_{j\rightarrow+\infty}\tau_j h_j^{\abs{\beta}}\partial^{\beta} p(y_j)}\frac{\partial_x^\alpha \inner{x^\beta}}{\beta!}.
   \end{eqnarray}
On the other hand, using the Taylor expansion 
   \begin{eqnarray*}
  \tau_j\big[p(y_j+h_jx)-p(y_j)\big]&=&\sum_{1\leq\abs{\beta}\leq r}\frac{\tau_j  h_j^{\abs{\beta}}\partial^{\beta} p(y_j)}{\beta !}x^{\beta},
   	\end{eqnarray*}
   	we have, for any $\abs\alpha\geq 1,$ 
   	\begin{eqnarray*}
   \tau_j	h_j^{\abs{\alpha}}\inner{\partial^\alpha p} (y_j+ h_j x)& =&\tau_j\partial_x^\alpha\Big( p(y_j+ h_j x)-  p(y_j)\Big)\\
   &=&\sum_{1\leq\abs{\beta}\leq r}\frac{  \tau_jh_j^{\abs{\beta}}\partial^{\beta} p(y_j)}{\beta !} \partial_x^\alpha\inner{x^{\beta}},   	\end{eqnarray*}
   	which along with \eqref{alq} yields
   	\begin{eqnarray*}
   	\forall~\abs\alpha\geq 1,\quad 	\lim_{j\rightarrow+\infty}	\tau_jh_j^{\abs{\alpha}}\inner{\partial^\alpha p} (y_j+ h_j x)=\partial^\alpha q(x).
   	\end{eqnarray*}
   In particular,  
   \begin{eqnarray}\label{+gra++}
   	\lim_{j\rightarrow+\infty}	\tau_jh_j \abs{\inner{\partial_x p} (y_j+ h_j x)}=\abs{\partial_x q(x)}
   \end{eqnarray}
   and
   \begin{equation}\label{+eige}
 \forall~ 1\leq\ell\leq n,	\quad \lim_{j\rightarrow+\infty}\tau_j	h_j^{2}  \lambda_{p,\ell} (y_j+ h_j x)=\lambda_{q,\ell}(x).
  	 \end{equation}
  	 Moreover observing $\mathcal M_q(x)=\sum_{1\leq \ell\leq n} \max\set{-\lambda_{q,\ell}(x),0}$  we have 
  	 \begin{eqnarray}\label{+max++}
   	\lim_{j\rightarrow+\infty}\tau_j	h_j^2  \mathcal M_p (y_j+ h_j x) = \mathcal M_q(x)
   \end{eqnarray}
   because of \eqref{+eige}.    
 For any  $\ell\in I_{x,q}$ with  $x\in\mathbb R^n$ given,  we see $\lambda_{q,\ell}(x)>0.$ Then  using \eqref{+eige} gives
  	 \begin{eqnarray*}
    \lambda_{p,\ell} (y_j+ h_j x)>0	
  	 \end{eqnarray*}
 for all $j$ large enough, since $\tau_j$ and $h_j$ are positive. Furthermore note  $y_j\rightarrow 0,h_j\rightarrow 0,$  and thus  for any $x\in\mathbb R^n$ we have  $y_j+h_jx\in B_\sigma $ for all $j$  large enough.   Consequently it follows from \eqref{lambdap} that,  for  all $j$ large enough, 
 \begin{eqnarray*}
  	   \lambda_{p,\ell} (y_j+ h_j x)\leq   C  \inner{ \mathcal M_p(y_j+ h_j x)+  \abs{\partial_x
  p(y_j+ h_j x)}^2}.
 \end{eqnarray*}
Combining the above estimate and \eqref{+gra++}-\eqref{+max++} we obtain 
 \begin{eqnarray*}
 	\lambda_{q,\ell}(x) &=& \lim_{j\rightarrow+\infty}\tau_j h_j^{2}  \lambda_{p,\ell} (y_j+ h_j x)\\
 	&\leq & C \lim_{j\rightarrow+\infty} \inner{\tau_j h_j^{2}\mathcal M_p(y_j+ h_j x)+ \tau_j h_j^{2}\abs{\partial_x
  p(y_j+ h_j x)}^2}
  \\&=& C  \inner{ \mathcal M_q( x)+  \abs{\partial_x
  q( x)}^2},
   \end{eqnarray*}
which holds for any  $\ell\in I_{x,q}.$  This gives the first statement \eqref{eq+1}  as desired.

 Next we prove the second statement. Let $q\in \mathcal L_{p,0}$ satisfy \eqref{eq+1}.   For the symmetric Hessian matrix $\inner{\partial_{x_ix_j} q(x)}_{1\leq i, j\leq  n}$, we can find  a $n\times n$   orthogonal matrix $Q(x)=\big(q_{ij}(x)\big)_{1\leq i,j\leq n}$ such that  
\begin{eqnarray}\label{diag}
	Q^T\begin{pmatrix}
		\lambda_{q,1}\\
		&\lambda_{q,2}\\
		&&\ddots\\
		&&&\lambda_{q,n}
	\end{pmatrix} Q=\inner{\partial_{x_ix_j}q}_{1\leq i,j\leq n},
\end{eqnarray}
 Define  $a_{ij}, b_{ij}, c_{ij}, 1\leq i,j\leq n,$ as follows. $b_{ij}=0$ if $i\neq j,$ and
\begin{eqnarray}\label{defb}
  b_{jj}(x)=\left\{
  \begin{array}{ll}
  \sqrt{\tilde C}, & \quad {\rm if}~~ \lambda_{q,j}(x)\leq 0,\\
  1, &\quad{\rm if}~~ \lambda_{q,j}(x)>0,
  \end{array}
  \right.
\end{eqnarray}
with $\tilde C\geq 1$  the  constant  in \eqref{eq+1}. And  \begin{eqnarray}\label{defofa}
  a_{ij}=\sum_{1\leq k\leq n}\inner{b_{kk}q_{ki}}\inner{b_{kk}q_{kj}}.
\end{eqnarray}
Then we can verify that, for any $\eta=\inner{\eta_1,\cdots,\eta_n}\in\mathbb R^n,$
\begin{eqnarray*}
	\sum_{1\leq i,j\leq n}a_{ij} \eta_i \eta_j &=&  \sum_{1\leq k\leq n} b_{kk}b_{kk}   \Big(\sum_{1\leq i\leq n} q_{ki}\eta_i \Big) \Big(\sum_{1\leq j\leq n} q_{kj}\eta_j \Big)   \\
	&\geq & \sum_{1\leq k\leq n}     \Big(\sum_{1\leq i\leq n} q_{ki}\eta_i \Big) \Big(\sum_{1\leq j\leq n} q_{kj}\eta_j \Big)\\
	&\geq & \sum_{1\leq i, j\leq n}     \Big(\sum_{1\leq k\leq n} q_{ki}     q_{kj}\Big)\eta_i\eta_j \\
	&=&    \abs\eta^2,
\end{eqnarray*}
the last line using the fact that $Q(x)$ is an
orthogonal. 
Thus $\inner{a_{ij}}_{n\times n}$ is a positive-definite  matrix.   Similarly  we use the relations \eqref{diag} and \eqref{defofa}  to compute, letting $\delta_{k\ell}$ be the the Kronecker delta function, 
\begin{eqnarray*}
&&	\sum_{1\leq i, j\leq n}a_{ij}(x)\partial_{x_i}\partial_{x_j} q(x)\\
&=&\sum_{1\leq i, j\leq n}\Big(\sum_{k=1}^n\inner{b_{kk}(x)q_{ki}(x)}\inner{b_{kk}(x)q_{kj}(x)}\Big)\Big(\sum_{  \ell=1}^nq_{\ell i}(x)q_{\ell j}(x)\lambda_{q,\ell}(x)\Big) \\
	&=&\sum_{1\leq k, \ell\leq n} b_{kk}(x) ^2\lambda_{q,\ell}(x)\Big(\sum_{  i=1}^nq_{ki}(x)q_{\ell i}(x)\Big) \Big(\sum_{j=1}^n q_{k j}(x)q_{\ell j}(x)\Big) \\
	&=&\sum_{1\leq k, \ell\leq n} b_{kk}(x) ^2\lambda_{q,\ell}(x)\delta_{k\ell}=\sum_{1\leq k\leq n} b_{kk}(x) ^2\lambda_{q,k}(x)\\
&=&	\sum_{ i\in I_{x,q}} \lambda_{q,i} (x)+\tilde C\sum_{ i\not\in I_{x,q}} \lambda_{q,i} (x),  
	\end{eqnarray*}
the last equality following from \eqref{defb}.   Now suppose $q$ satisfies \eqref{eq+1}.
 Then it follows from the above equalities that for any $x\in\mathbb R^n$ we have, observing $\mathcal M_q=-\sum_{ i\not\in I_{x,q}} \lambda_{q,i},$
	\begin{eqnarray*}
	\sum_{1\leq i, j\leq n}a_{ij}(x)\partial_{x_i}\partial_{x_j} q(x)  
	\leq \tilde C\big(\mathcal M_q(x)+\abs{\partial_x q(x)}^2\big)-\tilde C\mathcal M_q(x)= \tilde C\abs{\partial_x q(x)}^2.
\end{eqnarray*}
As a result, by maximum principle for elliptic equations we conclude that  $q$ can not  have any local minimum in  $\mathbb R^n$   unless  it is a constant. 	 Observe $q$ vanishes at $0.$ Thus any $q\in\mathcal L_{p,0}\setminus\set{0}$ can not  have any local minimum in  $\mathbb R^n.$  
 The proof of Lemma \ref{lemcr}  is thus complete.
 \end{proof}

As an immediate consequence of Theorem \ref{prohel} and  Lemma \ref{lemcr}, we have
 
 \begin{cor}\label{corcr}
 	Let $ p\in E_r.$  Suppose that there are two constants $C, \sigma>0$   such that   
\begin{equation*}
	\forall~x\in B_{\sigma},\quad	\sum_{j\in I_{x,p}} \lambda_{p,j}(x)\leq  C  \inner{\mathcal M_p(x)+ \abs{\partial_x
  p(x)}^2},
\end{equation*}
where we use the notations given in \eqref{Ixprho}- \eqref{bal}.  Then we can find a    constants  $c_0>0$ depending only on $p$   such that, decreasing $\sigma$ if necessary, 
\begin{eqnarray*}
  \norm{\partial_{x}u}^2_{L^2}
    +\tau^2 \norm{ \inner{\partial_{x}p} u}^2_{L^2}\leq
    c_0 \norm{\big(\partial_{x}+ \tau\inner{\partial_{x} p}  \big)u}_{L^2 }^2+c_0\norm{u}_{L^2}^2
	\end{eqnarray*}
	holds for any $u\in C_0^\infty(B_\sigma)$ and for any $\tau>0.$
	By density arguments we see the above estimate still holds for any $u\in H_0^1(B_\sigma)$  with $H_0^1(B_\sigma)$ the classical Sobolev space.   Note that $\norm{\big(\partial_{x}+ \tau\inner{\partial_{x} p}  \big)u}_{L^2 }^2= \big(\Delta_{\tau p}^{(0)}u,~u\big)_{L^2}$ if $u\in  C_0^\infty(B_\sigma).$ 
	 \end{cor}

\subsection{Localization} 
Here we  introduce   some partitions of unity related to a slowly varying metric.    Recall a metric $g$ is slowly varying if we can find two constant $C\geq 1$ and $r>0$
such that 
\begin{equation}\label{property}
	\forall~x,y\in\mathbb{R}^n,~~\quad g_{x}\inner{y-x}\leq r^2 \Longrightarrow
   \forall~T\in\mathbb R^n,\quad  C^{-1} \leq \frac{g_{x}(T)}{g_{y}(T)}\leq C.
\end{equation}
And we refer to   \cite{Hormander85,MR2599384} for more
details on the metrics and the symbol space related to a metric.

\begin{rem}\label{rem1}
In order to prove a metric $g$ is slowly varying we  ask only that 
 \begin{eqnarray*}
 	\exists~C\geq 1,\quad\exists~r>0,\quad  \forall~x,y,T\in\mathbb{R}^n,\quad g_x\inner{y-x}\leq r^2\Longrightarrow
  g_y(T) \leq Cg_x(T),
 \end{eqnarray*}
 which is sufficient to give the previous property \eqref{property}, see \cite[Remark 2.2.2]{MR2599384} for instance. 
\end{rem}

Now we define   $f$ by setting 
\begin{equation}\label{fv}
	f(x)= \sum_{1\leq \abs \alpha \leq k}\inner{1+\abs{\partial^\alpha V(x)}^2}^{\frac{1}{2\abs\alpha}},
\end{equation}
which is a regularization of the function $\tilde f$ in \eqref{defoffv}. Observe $\tilde f \leq f \leq C_k \inner{1+\tilde f}$ for some constant $C_k$ depending only on $k.$  
Let $V$ be the potential satisfying   \eqref{conupp} in Assumption \ref{assu}. Then   for any multi-index $\alpha$ with $\abs\alpha=k+1$ we can find a constant $C_\alpha$ depending on $\alpha$ such that 
\begin{equation}\label{weakcond}
	\forall~x\in\mathbb R^n, \quad  \abs{\partial^\alpha V(x)}	\leq C_\alpha f(x)^{k+1-\delta_2},
\end{equation}
with $\delta_2>0$  the arbitrarily small number given in \eqref{conupp}.      Moreover  letting $\eps>0$ be a small number to be determined further, we define the metric $g_\eps$ as follows.
\begin{equation}\label{7292}
  g_{x,\eps}=\eps f(x)^{2}\abs{dx}^2.
\end{equation}
Next we will show that the metric defined above is slowly varying. 

\begin{lem}\label{lem729}
Let $V$ be the potential satisfying the condition \eqref{weakcond}. Then the   metric defined by \eqref{7292} is slowly varying,  i.e., we can find two constants $C_*,r>0,$ depending only the constants in \eqref{weakcond} but independent of $\eps$,  such
  that if $g_{x,\eps}(y-x)\leq \eps r^2$ then $$C_*^{-2}\leq \frac{g_{x,\eps}}{g_{y,\eps}}\leq
  C_*^2.$$
\end{lem}

In order to prove the above lemma we need the  bootstrap principle.  Here we refer to \cite[Proposition 1.21]{tao}.

\begin{prop}\label{probootst}
For each $x\in\mathbb R^n$ we have two statements, a ``hypothesis" $\mathbf H(x)$ and  a ``conclusion" $\mathbf C(x)$, with the following assertions listed subsequently fulfilled. 	
\begin{enumerate}[(i)]
		\item If   $\mathbf{C}(x)$ is true for some $x_0$ then   $\mathbf{H}(x)$ holds for all $x$ in a neighborhood of  $x_0$.
		\item If $x_j, j\geq 1,$ is a sequence in $\mathbb R^n$ which converges to some $\tilde x$, and if   $\mathbf{C}(x_j)$ is true for all $ j\geq 1,$ then   $\mathbf{C}(\tilde x)$ is true.
		\item  $\mathbf{H}(x)$ is true for at least one  $x\in\mathbb R^n.$
		\item  If $		 \mathbf{H}( x)$ is true for some $x\in\mathbb R^n$ then so is $\mathbf{C}( x) $  for the same $x.$
	\end{enumerate}
	Then   $\mathbf{C}( x)$ is true for all $x\in\mathbb R^n.$
\end{prop}

The proof of the proposition above is just the same as that in \cite[Proposition 1.21]{tao},  with  the time interval $I$ therein replaced by $\mathbb R^n.$

\begin{proof}[Proof of Lemma \ref{lem729}]
Note that 
\begin{eqnarray*}
 g_{x,\eps}(y-x)\leq \eps r^2 \Longleftrightarrow \abs{y-x}\leq   f(x)^{-{1}} r.
\end{eqnarray*}
 Then in view of Remark \ref{rem1}  we only need show that 
\begin{equation}\label{slow+}
    \exists~ r,C_*>0,~\forall~x,y\in\mathbb{R}^n,~~\quad \abs{y-x}\leq r   f(x)^{-{1}} \Longrightarrow
  f(y)\leq  C_* f(x).
  \end{equation}
To do so we use bootstrap arguments stated in Proposition \ref{probootst}.  Let $0<r<1$ to be determined later.  We define  a continuous function $\psi_r(x)$ by setting 
\begin{eqnarray*}
		\psi_r(x)=\max_{z\in\big\{\abs{z-x}\leq r f(x)^{-{1}}\big\}} \frac{f(z)}{f(x)}.
	\end{eqnarray*}
Let $ C_*>\psi_1(0)+1$ be a parameter to be chosen later, and let $\mathbf{H}(x)$	denote the statement that
\begin{eqnarray*}
	\psi_r(x)\leq 2C_*
\end{eqnarray*}
and let  $\mathbf{C}(x)$	denote the statement that
\begin{eqnarray*}
	\psi_r(x)\leq C_*.
\end{eqnarray*}
The continuity of $\psi_r$ gives the following assertions: 
	\begin{enumerate}[(i)]
		\item If   $\mathbf{C}(x)$ is true for some $x_0$ then   $\mathbf{H}(x)$ holds for all $x$ in a neighborhood of  $x_0$.
		\item If $x_j, j\geq 1,$ is a sequence in $\mathbb R^n$ which converges to some $\tilde x$, and if   $\mathbf{C}(x_j)$ is true for all $ j\geq 1,$ then   $\mathbf{C}(\tilde x)$ is true.
		\item  $\mathbf{H}(0)$ is true, recalling  $C_*>\psi_1(0)\geq \psi_r(0)$ for $0<r<1$.
	\end{enumerate}
Then by Proposition \ref{probootst} we see $\mathbf{C}(x)$ will be true for all $x\in\mathbb R^n$ if we can show that 
\begin{equation}\label{last sta}
	\begin{aligned}
		 \mathbf{H}( x) ~\textrm{is true for some} ~x\in\mathbb R^n\Longrightarrow  \mathbf{C}( x) ~\textrm{is also true for the same} ~x.
	\end{aligned}
\end{equation}

In the following arguments we will prove the property \eqref{last sta} under the hypothesis in \eqref{slow+}.  We will use $C_j \geq 1,  j\geq 1,$ to denote different constants which  depend only on the integer $k$ and the constants given in \eqref{weakcond}.   For any $\alpha$ with $1\leq \abs\alpha\leq k$  we have the Taylor expansion for $\partial^\alpha V:$ 
	\begin{equation}\label{est}
	\begin{aligned}
		\partial^\alpha V(y)&=\sum_{\abs\beta\leq k-\abs\alpha}\frac{\partial^{\beta+\alpha}V(x)}{\beta!}(y-x)^\beta\\
		+&\sum_{\abs\beta= k+1-\abs\alpha} \frac{\abs\beta}{\beta!}(y-x)^\beta\int_0^1(1-\theta)^{\abs\beta-1}\partial^{\beta+\alpha} V\inner{x+\theta(y-x)}d\theta.
		\end{aligned}
	\end{equation}
	  From  the definition of $f$ and the fact  $0<r<1$  it follows that   if $\abs{y-x}\leq r   f(x)^{-{1}}$ then
	\begin{equation}\label{est1}
		\Big|\sum_{\abs\beta\leq k-\abs\alpha}\frac{\partial^{\beta+\alpha}V(x)}{\beta!}(y-x)^\beta \Big|\leq  C_1f(x)^{\abs\alpha}.
	\end{equation}
Moreover for the last term in 
\eqref{est}, we  use \eqref{weakcond} to compute, supposing  $\abs{y-x}\leq r   f(x)^{-{1}},$ 
	\begin{eqnarray*}
		&&\Big|\sum_{\abs\beta= k+1-\abs\alpha} \frac{\abs\beta}{\beta!}(y-x)^\beta\int_0^1(1-\theta)^{\abs\beta-1}\partial^{\beta+\alpha} V\inner{x+\theta(y-x)}d\theta\Big|\\
		&\leq& C_2 \sum_{\abs\beta= k+1-\abs\alpha}    r^{\abs\beta} f(x)^{-\abs\beta}   \int_0^1   \Big(f\inner{x+\theta(y-x)}\Big)^{k+1-\delta_2} d\theta \\
		&\leq&r C_2  f(x)^{\abs\alpha-\delta_2} \sum_{\abs\beta= k+1-\abs\alpha}    \int_0^1  \Big(\frac{f\inner{x+\theta(y-x)}}{f(x)}\Big)^{k+1-\delta_2} d\theta \\
		&\leq&r C_3 f(x)^{\abs\alpha} \psi_r(x)^{k+1-\delta_2},
	\end{eqnarray*}
	the second inequality using the fact that $0<r<1$ and the last inequality following from the definition of $\psi_r.$  As a result the validity of $\mathbf{H}( x)$ gives that 
	\begin{eqnarray*}
		&&\Big|\sum_{\abs\beta= k+1-\abs\alpha} \frac{\abs\beta}{\beta!}(y-x)^\beta\int_0^1(1-\theta)^{\abs\beta-1}\partial^{\beta+\alpha} V\inner{x+\theta(y-x)}d\theta\Big|\\
		&\leq& rC_3 f(x)^{\abs\alpha} (2C_*)^{k+1-\delta_2},
	\end{eqnarray*}
	which along with \eqref{est}-\eqref{est1} yields that  for any $\alpha$ with $1\leq \abs\alpha\leq k$ and for any $y\in\mathbb R^n$ with $\abs{y-x}\leq r   f(x)^{-{1}},$ 
 we have
	\begin{eqnarray*}
		\abs{\partial^\alpha V(y)}\leq C_1f(x)^{\abs\alpha}+rC_3 f(x)^{\abs\alpha} (2C_*)^{k+1-\delta_2}
	\end{eqnarray*}
and thus,  observing $C_1, C_3, C_*\geq 1,$
\begin{eqnarray*}
		\abs{\partial^\alpha V(y)}^{1/\abs\alpha}&\leq& C_1 f(x) +r^{1/\abs\alpha}C_3 f(x) (2C_*)^{(k+1-\delta_2)/\abs\alpha}\\
		&\leq& C_1 f(x) +r^{1/k}C_3 f(x)  (2C_*)^{k+1}.
	\end{eqnarray*}
	This implies, in view of the definition of $f,$ 
	\begin{eqnarray*}
	f(y)	\leq C_4  f(x) +r^{1/k}C_4 f(x)  (2C_*)^{k+1},
	\end{eqnarray*}
	that is,
	\begin{eqnarray*}
		\frac{f(y)}{f(x)}	\leq C_4   +r^{1/k}C_4    (2C_*)^{k+1}.
	\end{eqnarray*}
 Observe the above inequality  holds for all $y$ such that $\abs{y-x}\leq r   f(x)^{-{1}}$ and thus  
 \begin{equation}\label{fiest}
 	\psi_r(x)\leq  C_4   +r^{1/k}C_4    (2C_*)^{k+1}.
 \end{equation}
 Now we choose $C_*$  such that
 \begin{eqnarray*}
 	C_*=2C_4+\psi_1(0)+1
 \end{eqnarray*}
 and choose  such a  small $r$ that
 \begin{eqnarray*}
 	r^{1/k}     (2C_*)^{k+1}\leq 1.
 \end{eqnarray*}
 Then \eqref{fiest} gives $\psi_r(x)<2C_4<C_*$ and thus $\mathbf{C}(x)$ holds, completing the proof of the property \eqref{last sta}.   As a result we use Proposition \ref{probootst} to conclude that 
 \begin{eqnarray*}
 	\psi_r(x)\leq C_*
 \end{eqnarray*}
 for all $x\in\mathbb R^n, $  with the constants $C_*$ and $r$ chosen above.  This  yields the assertion \eqref{slow+} as desired, completing the proof of Lemma \ref{lem729}.
\end{proof}

Let $g_\eps$ be the   metric given by \eqref{7292}.  We denote by $S(1,g_\eps)$
the class of smooth real-valued functions $a(x)$ satisfying
  the following condition:
  \begin{eqnarray*}
    \forall~\gamma\in \mathbb{Z}_+^n,~~\forall~x\in\mathbb{R}^n,\quad
    \abs{\partial^\gamma a(x)}\leq C_\gamma \inner{\eps^{1/2}f(x)}^{ \abs\gamma},
  \end{eqnarray*}
  with $C_\gamma$ the constants depending only on $\gamma,$ but independent of $\eps.$
  The space $S(1,g_\eps)$ endowed with the seminorms
  \[
    \abs{a}_{\ell,S(1,g_\eps)}=\sup_{x\in\mathbb{R}^n,\abs{\gamma}=\ell}\inner{\eps^{1/2}f(x)}^{-\abs\gamma}\abs{\partial^\gamma a(x)},\quad
    \ell\geq 0,
  \]
  becomes a Fr\'{e}chet space.

The main feature of a slowly varying metric is that it enables us to
introduce some partitions of unity related to the metric.  We can apply 
\cite[Lemma 1.4.9 and Theorem 1.4.10]{Hormander85} to $\norm{y}_x=\inner{g_{x,\eps}(y)/(\eps r^2)}^{1/2}$ with $r$ the number given in Lemma \ref{lem729}; this gives 
  the following lemma (see also \cite[Lemma 18.4.4]{Hormander85} with $c$ therein replaced by $\eps r^2$).

\begin{lem}[Partition of unity]\label{uni}
  Let $g_\eps$ be the   metric given by \eqref{7292} and let $r, C_*$ be the constants given in Lemma \ref{lem729}.  We can find   a sequence $x_\mu\in\mathbb{R}^n,\mu\geq1,$ such
  that the union of the balls
  \[
    \Omega_{\mu,\eps,r}=\set{x\in\mathbb{R}^{n};\quad g_{x_\mu,\eps}\inner{x-x_\mu}<\frac{\eps r^2}{2}}
  \]
  covers  the  whole space $\mathbb{R}^{n}.$ Moreover there exists a positive integer
  $N ,$ depending only  on $C_*$ and the dimension  $n$ but independent of $\eps,$  such that the intersection
  of more than $N $ balls is always empty.
  One can choose a
  family of nonnegative functions $\set{\varphi_{\mu,\eps}}_{\mu\geq 1}$ uniformly bounded in
  $S(1,g_\eps)$ with respect to $\mu,$ such that
 \begin{equation*}
   {\rm supp}~ \varphi_{\mu,\eps}\subset\Omega_{\mu,\eps,r},\quad
   \sum_{\mu\geq1} \varphi_{\mu,\eps}^2 =1~~\,\,{\rm and }~~\,\,
   \sup_{\mu\geq1}\abs{\partial_x \varphi_{\mu,\eps} (x)}\leq C\eps^{1/2} f(x),
\end{equation*}
where $C$ is a constant independent of $\eps.$
  Here by  uniformly bounded in
  $S(1,g_\eps)$ with respect to $\mu,$ we mean
  \begin{eqnarray*}
     \sup_{\mu}\abs{\varphi_{\mu,\eps}}_{\ell,S(1,g_\eps)}\leq C_{\ell}, \quad \ell\geq
     0,
  \end{eqnarray*}
  with $C_\ell$ constants depending only on $\ell.$
\end{lem}

\begin{rem}\label{7308}
It follows from Lemma \ref{lem729} that for any
  $\mu\geq1$ one has
\begin{equation}\label{7310}
  \forall~x,y\in \Omega_{\mu,\eps,r},\quad C_*^{-1}f(y) \leq f(x) \leq C_* f(y),
\end{equation}
where $C_*$ is the constant given in Lemma \ref{lem729}.
\end{rem}

 \subsection{Proof of Theorem \ref{thmmain}} This part is devoted to proving Theorem \ref{thmmain}, and we only need to prove  the estimates \eqref{cesc} and  \eqref{lastestimate}, and the compactness of the resolvent for Witten Laplacian will follow immediately from these estimates due to (iii) in Assumption \ref{assu}.  In the proof  we let $\tilde f,  f$ be the functions introduced respectively in \eqref{defoffv} and \eqref{fv}, satisfying that
\begin{eqnarray}\label{refv}
	\tilde f \leq f \leq C_k \inner{1+\tilde f}
\end{eqnarray}
for some constant $C_k$ depending only on $k.$  
Recall $ E_k$ is the set of polynomials with degree less than or equal to $k,$ and $B_\sigma$ denotes the ball centered at $0$ with radius $\sigma.$

\begin{proof}[Proof of  Theorem \ref{thmmain} (Maximal estimate)] 
We begin with the first assertion \eqref{cesc} and  will prove it  by  contradiction.  To do so   suppose that, contrary to  \eqref{cesc}, for any $\ell \geq 1$ and for any $\tau>0,$  there exists a function $u_\ell=u_{\ell,\tau}\in C_0^\infty(\mathbb R^n)$ with $u_\ell \not\equiv 0,$  such that   
\begin{eqnarray}\label{upes}
     \norm{\tilde f u_\ell}_{L^2 }^2
      > \ell
      \inner{\triangle_{ \tau V}^{(0)} u_\ell,~u_\ell}_{L^2 }+ \ell\norm{u_\ell}_{L^2 }^2.
\end{eqnarray} 	
Here and throughout the proof we will write $u_{\ell}$ instead of $u_{\ell,\tau},$ omitting the dependence of $\tau,$ to simplify the notation. 
For given     $0<\eps\leq 1$  to be determined further, we let  $\set{\varphi_{\mu,\eps}}_{\mu\geq 1}$ be the partition of unity given  in Lemma \ref{uni}, which satisfies that  
\begin{eqnarray*}
	{\rm supp}\,\varphi_{\mu,\eps}\subset \Omega_{\mu,\eps,r}=\Big\{x\in\mathbb R^n; \quad  \abs{x-x_{\mu}}< \frac{r}{\sqrt 2 f(x_{\mu})}\Big\}
\end{eqnarray*}
with  $r>0$ the number given in Lemma  \ref{lem729} and that  
\begin{equation}\label{unup}
\abs{\partial_x \varphi_{\mu,\eps}}\leq \tilde C_*\,\eps^{1/2} f	
\end{equation}
with $\tilde C_*$ a constant independent of $\eps$ and $\mu.$   To simplify the notations  we will use $C_j, j\geq 1,$ in the following discussions to denote the suitable constants which depend on $\tilde C_*$ above but are independent of $\eps, \tau, \mu$ and $\ell.$   By  the IMS localization formula (cf. \cite[Theorem 3.2]{cfks})  we obtain
 \begin{multline*}
 	 \inner{\triangle_{\tau V}^{(0)} u_\ell,~u_\ell}_{L^2 } =  \sum_{\mu\geq 1} \inner{\triangle_{\tau V}^{(0)}\inner{ \varphi_{\mu,\eps} u_\ell},~ \varphi_{\mu,\eps} u_\ell}_{L^2 }-\sum_{\mu\geq 1}\norm{ \inner{\partial_x\varphi_{\mu,\eps}}  u_\ell}_{L^2}^2\\
 	  \geq  \sum_{\mu\geq 1} \inner{\triangle_{\tau V}^{(0)}\inner{ \varphi_{\mu,\eps} u_\ell},~ \varphi_{\mu,\eps} u_\ell}_{L^2 }-C_1 \,\eps \sum_{\mu\geq 1}\norm{f(x)  \varphi_{\mu,\eps}   u_\ell}_{L^2}^2,
 \end{multline*}
 where the last inequality follows from \eqref{unup} and the fact that  the intersection
  of more than $N$ balls $\Omega_{\mu, \eps,r}$ is always empty with $N$ a fixed integer given in Lemma \ref{uni} independent of $\eps$.   As a result we combine \eqref{upes} and the above estimate to conclude   
 	 	\begin{multline*}
 \sum_{\mu\geq 1}  \norm{ \tilde f  \varphi_{\mu,\eps}u_\ell}_{L^2 }^2\\
      >   \sum_{\mu\geq 1}   \ell  \bigg[\inner{\triangle_{\tau V}^{(0)}\inner{ \varphi_{\mu,\eps} u_\ell},~ \varphi_{\mu,\eps} u_\ell}_{L^2 }+  \norm{ \varphi_{\mu,\eps}   u_\ell}_{L^2}^2-C_1\,\eps    \norm{f(x)  \varphi_{\mu,\eps}   u_\ell}_{L^2}^2\bigg].
    \end{multline*}
    Thus for any $\ell$,  there exists a positive integer $\mu_\ell,$ depending only on $\ell,$ such that 
 	  \begin{multline*} 
    \norm{ \tilde f \varphi_{\mu_\ell,\eps} u_\ell}_{L^2 }^2 \\
      >   \ell  \bigg[\inner{\triangle_{\tau V}^{(0)}\inner{ \varphi_{\mu_\ell,\eps} u_\ell},~ \varphi_{\mu_\ell,\eps} u_\ell}_{L^2 }+  \norm{ \varphi_{\mu_\ell,\eps}   u_\ell}_{L^2}^2- C_1\,\eps    \norm{f(x)  \varphi_{\mu_\ell,\eps}   u_\ell}_{L^2}^2\bigg].
 	    \end{multline*}
 	   As a result  we use  \eqref{refv} to conclude that, for all $\ell$ large enough such that $  \ell>1/\eps,$ 
 	    \begin{eqnarray} \label{evar}
 	    \begin{split}
    0 
      > &   \ell  \bigg[\inner{\triangle_{\tau V}^{(0)}\inner{ \varphi_{\mu_\ell,\eps} u_\ell},~ \varphi_{\mu_\ell,\eps} u_\ell}_{L^2 }+\inner{1-\eps C_2}  \norm{ \varphi_{\mu_\ell,\eps}   u_\ell}_{L^2}^2\\
      &\qquad-\eps\,  C_2     \norm{\tilde f(x)  \varphi_{\mu_\ell,\eps}   u_\ell}_{L^2}^2\bigg],
      \end{split}
 	    \end{eqnarray}
with 
\begin{eqnarray*}
 {\rm supp}\,  (\varphi_{\mu_\ell,\eps}   u_\ell)\subset\Omega_{\mu_\ell,\eps,r}= \Big\{x\in\mathbb R^n; \quad  \abs{x-x_{\mu_\ell}}< \frac{r}{\sqrt 2 f(x_{\mu_\ell})}\Big\}.  
 \end{eqnarray*}
 We claim  that there exists a subsequence $\big\{x_{\mu_{\ell_j}}\big\}_{j\geq 1} $ of $  x_{\mu_\ell} $  such that 
\begin{equation}\label{beinf}
\lim_{j\rightarrow+\infty}  |x_{\mu_{\ell_j}}| =+\infty. 
\end{equation}
Otherwise we can find a constant $R>0$ such that
\begin{eqnarray*}
	\forall~\ell\geq 1,\quad |x_{\mu_\ell}|\leq R,
\end{eqnarray*}
which yields, using the notation $M_R\stackrel{\rm def}{ =}\max_{\abs{x}\leq R} f(x)$ and observing $f\geq 1,$  
\begin{eqnarray}\label{suppv}
     \Big\{ x;~  \abs{x-x_{\mu_\ell}} <\frac{  r}{\sqrt 2 M_R} \Big\} \subset \Omega_{\mu_\ell, \eps,r}=\Big\{x; ~  \abs{x-x_{\mu_\ell}}< \frac{r}{\sqrt 2 f(x_{\mu_\ell})}\Big\}
\end{eqnarray}
and
\begin{eqnarray*}
	  \bigcup_{\ell\geq 1} \Omega_{\mu_\ell, \eps,r}\subseteq \Big\{x\in\mathbb R^n; \quad  \abs{x}< R+ 2^{-1/2}  r  \Big\}.
\end{eqnarray*}
We then have a contradiction, since 
the Lebesgue measure of the set on the right hand side is finite and independent of $\ell$, meanwhile  the Lebesgue measure of  
\begin{eqnarray*}
	 \bigcup_{\ell\geq 1} \Omega_{\mu_\ell, \eps,r} 
\end{eqnarray*}
is $+\infty$ due to \eqref{suppv} and the fact that the intersection
  of more than $N$ balls $\Omega_{\mu_\ell, \eps,r}$ is always empty. The contradiction implies the conclusion \eqref{beinf} and thus
    \begin{equation}\label{longbe}
 \lim_{j\rightarrow+\infty} f(x_{\mu_{\ell_j}})=+\infty	
  \end{equation}
 because of   the statement  (iii) in Assumption \ref{assu}.

Now we denote 
\begin{eqnarray*}
v_j=\varphi_{\mu_{\ell_j},\eps}u_{\ell_j},\quad j\geq 1.
\end{eqnarray*}
Then
\begin{eqnarray*}
{\rm supp}\, v_j \subset  \Big\{x\in\mathbb R^n; \quad  \abs{x-x_{\mu_{\ell_j}}}< \frac{r}{\sqrt 2 f(x_{\mu_{\ell_j}})}\Big\},
\end{eqnarray*}
and furthermore,   
in view of  \eqref{evar}, 
\begin{eqnarray} \label{cruest}
  0
   >     \inner{\triangle_{\tau V}^{(0)}v_j ,~ v_j}_{L^2 }+ \inner{1-\eps\,C_2} \norm{ v_j}_{L^2}^2- \eps \,C_2   \norm{\tilde f(x)  v_j}_{L^2}^2.
\end{eqnarray}
In the following discussion we will derive a contradiction through several steps, starting from the estimate \eqref{cruest}.

\bigskip
\noindent \underline{{\it Step 1}}.  
We define 
\begin{eqnarray*}
	w_j(x)=v_j\inner{x_{\mu_{\ell_j}}+f(x_{\mu_{\ell_j}})^{-1}x}.
\end{eqnarray*}
Then  $w_j\in C_0^\infty(\mathbb R^n)$ with
\begin{eqnarray*}
{\rm supp}\,w_j\subset \Big\{x\in\mathbb R^n; ~\abs{x}< r/\sqrt 2\Big\},\quad j\geq 1.
\end{eqnarray*}
Using the changes of variable $x= x_{\mu_{\ell_j}}+f(x_{\mu_{\ell_j}})^{-1}y$ in \eqref{cruest} for  the $L^2$ integration,  we obtain
    \begin{equation}\label{cruest1}
    	0
   > \inner{\triangle_{\tau q_j}^{(0)}w_j ,~ w_j}_{L^2 }+  \frac{1-\eps\, C_2}{ f(x_{\mu_{\ell_j}})^2 }\norm{ w_j}_{L^2}^2- \eps \,C_2  \norm{\tilde f_{\tau q_j} w_j}_{L^2}^2,
    \end{equation}
 where
  \begin{eqnarray}\label{qofj}
 	  	q_j(x)=V\big(x_{\mu_{\ell_j}}+f(x_{\mu_{\ell_j}})^{-1}x\big)-V(x_{\mu_{\ell_j}})
 	  \end{eqnarray} 
 	    and
 \begin{eqnarray*}
 	  	\tilde f_{\tau q_j}(x)=\sum_{1\leq \abs\alpha \leq k} \tau^{1/\abs\alpha}\abs{ \partial^\alpha  q_j(x)}^{1/\abs\alpha}.
 \end{eqnarray*}
The inequality \eqref{cruest1} implies
\begin{eqnarray*} 
  \norm{w_j}_{L^2}+\norm{\partial_x w_j}_{L^2}
     >0 
  \end{eqnarray*} and thus we can define  
    \begin{eqnarray*}
    	\zeta_j=\frac{w_j}{\inner{\norm{w_j}_{L^2}^2+\norm{\partial_x w_j}_{L^2}^2}^{1/2}}.
    \end{eqnarray*}
    Then we have, recalling   $B_{\sigma}= \big\{x\in\mathbb R^n; \quad  \abs{x}<  \sigma  \big\},$  
    \begin{eqnarray}\label{uninor}
    \zeta_j\in C_0^\infty(B_{r/\sqrt2}), \qquad  \norm{\zeta_j}_{L^2}^2+\norm{\partial_x\zeta_j}_{L^2}^2=1,
    \end{eqnarray}
    and, dividing both sides of 
   \eqref{cruest1} by the factor $ \norm{w_j}_{L^2}^2+\norm{\partial_x w_j}_{L^2}^2,$
    \begin{equation*}
    	0
   >  \inner{\triangle_{\tau q_j}^{(0)}\zeta_j ,~ \zeta_j}_{L^2 }+  \frac{1-\eps\, C_2}{ f(x_{\mu_{\ell_j}})^2 }\norm{ \zeta_j}_{L^2}^2- \eps \,C_2 \norm{\tilde f_{\tau q_j}  \zeta_j}_{L^2}^2.
    \end{equation*}
Thus    \begin{equation}\label{cruc1}
     \liminf_{j\rightarrow+\infty}\Big[ \inner{\triangle_{\tau q_j}^{(0)}\zeta_j ,~ \zeta_j}_{L^2 } - \eps \,C_2 \norm{\tilde f_{\tau q_j} \zeta_j}_{L^2}^2\Big] 
 	 \leq   0,
    \end{equation}
  since  it follows from \eqref{longbe} that 
  \begin{eqnarray*}
  	 \frac{1-\eps\, C_2}{ f(x_{\mu_{\ell_j}})^2 }\norm{ \zeta_j}_{L^2}^2\rightarrow 0~\textrm{as}~\rightarrow +\infty.
  \end{eqnarray*}
  
  \bigskip
\noindent  \underline{{\it Step 2}}. Let $q_j$ be given in \eqref{qofj} with $V$ satisfying Assumption \ref{assu}.    We will prove here  that there exists a subsequence of $q_j,$ still denoted by $q_j$, and a polynomial $p\in E_k\setminus\set{0} ,$ such that 
 \begin{equation}\label{conv}
\forall~0\leq \abs\beta\leq k,~  \forall~ x\in B_{r/\sqrt 2},  \quad\lim_{j\rightarrow+\infty} \partial^\beta q_j(x) =\partial^\beta p (x),	
 \end{equation}
 and that, 
  using the notations given in \eqref{Ixprho} and \eqref{mrho}, 
  \begin{equation}\label{conforp}
	\forall~x\in B_{r/\sqrt 2},\quad	\sum_{j\in I_{x,p}} \lambda_{p,j}(x)\leq  C_3 \inner{\mathcal M_p(x)+ \abs{\partial_x
  p(x)}^2}.
\end{equation}

We begin with the proof of \eqref{conv}.  To do so we  
use Taylor's expansion to write
 	  \begin{eqnarray}\label{extaylor}
 	  \begin{split}
 	  	&q_j(x)=\sum_{1\leq\abs{\alpha}\leq k}\frac{  f(x_{\mu_{\ell_j}})^{-\abs{\alpha}}\partial^{\alpha} V(x_{\mu_{\ell_j}})}{\alpha !}x^{\alpha}\\
 	  	&+   f(x_{\mu_{\ell_j}})^{-(k+1)}  \sum_{\abs\alpha=k+1} \frac{\abs\alpha x^{\alpha}}{\alpha !} \int_{0}^{1} (1-\theta)^{k}\partial^{\alpha}  V\big(x_{\mu_{\ell_j}}+   \theta x f(x_{\mu_{\ell_j}})^{-1}  \big)  d\theta.
 	  	\end{split}
 	  \end{eqnarray}
 	  Moreover observe $f(x_{\mu_{\ell_j}})^{-\abs{\alpha}}\partial^{\alpha} V(x_{\mu_{\ell_j}}), 1\leq\abs\alpha\leq k,$  is an uniformly bounded sequence with respect to $j$ and thus we can find a  subsequence, still denoted by  $f(x_{\mu_{\ell_j}})^{-\abs{\alpha}}\partial^{\alpha} V(x_{\mu_{\ell_j}})$, such that 
 	\begin{equation}\label{alapha+}
 		\forall~1\leq \abs\alpha\leq k,\quad \lim_{j\rightarrow+\infty} f(x_{\mu_{\ell_j}})^{-\abs{\alpha}}\partial^{\alpha} V(x_{\mu_{\ell_j}}) =A_\alpha.  	
 		 \end{equation}
 	This gives  
 	\begin{eqnarray*}
 		\sum_{1\leq\abs\alpha\leq k} \abs{A_\alpha}^{1\over\abs\alpha}=\lim_{j }  f(x_{\mu_{\ell_j}})^{-1} \sum_{1\leq\abs\alpha\leq k}\abs{ \partial^{\alpha} V(x_{\mu_{\ell_j}}) }^{1\over\abs\alpha}=\lim_{j\rightarrow \infty}  \frac{\tilde f(x_{\mu_{\ell_j}})}{f(x_{\mu_{\ell_j}})}  >0,
 	\end{eqnarray*}
 	 	the last inequality using   \eqref{refv}  and the fact that $\tilde f(x_{\mu_{\ell_j}})\rightarrow+\infty$ as $j\rightarrow+\infty.$  As a result, defining  $p$ by
 	 	\begin{equation}\label{defofp+}
 	 		p=\sum_{1\leq \abs\alpha\leq k} \frac{A_\alpha }{\alpha!}x^\alpha, 
 	 	\end{equation}
 	 	we see $p\in E_k\setminus\big\{0\big\}$ and  the first term on the right side of \eqref{extaylor} converges to $p(x).$
 In order to treat the remainder term in \eqref{extaylor} we use \eqref{conupp}  and \eqref{7310} to obtain that,   for any $\gamma$ with $   \abs\gamma= k+1,$  and for any   $\abs{x}<   r/\sqrt 2$ and any $\theta\in[0,1],$
   \begin{eqnarray*}
   	\abs{\partial^{\gamma}  V\inner{x_{\mu_{\ell_j}}+   \theta x f(x_{\mu_{\ell_j}})^{-1}  }} 
    &\leq&    C_\gamma \Big[f\big(x_{\mu_{\ell_j}}+   \theta x f(x_{\mu_{\ell_j}})^{-1}  \big)\Big] ^{k+1-\delta_2}\\
   	 &\leq&   \tilde C_{\gamma}  f(x_{\mu_{\ell_j}})^{k+1-\delta_2},   \end{eqnarray*}
   	where  $\delta_2>0$ is an arbitrary small number and  $C_{\gamma},\tilde C_{\gamma}$ are two constants depending only on $\gamma.$  
   	This implies, for any $x\in B_{r/\sqrt 2},$
   \begin{eqnarray*}
   	  && f(x_{\mu_{\ell_j}})^{-(k+1)}  \sum_{\abs\alpha=k+1} \bigg|  \frac{\abs\alpha x^{\alpha}}{\alpha !} \int_{0}^{1} (1-\theta)^{k} \partial^{\alpha}  V(x_{\mu_{\ell_j}}+  \theta  xf(x_{\mu_{\ell_j}})^{-1}  ) d\theta  \bigg|  \\
   	  &  \leq &  C_4  f(x_{\mu_{\ell_j}})^{-\delta_2}\longrightarrow 0, ~~{\rm as}~~  j\rightarrow+\infty,
   	  \end{eqnarray*}
   	  the last line using \eqref{longbe}.
   	    As a result we have
   	   \begin{eqnarray*}
   	   	\forall~ x\in B_{r/\sqrt 2},  \quad\lim_{j\rightarrow+\infty}   q_j(x) = p (x).
   	   \end{eqnarray*}
 Similarly,	  
 	    for any $1\leq \abs\beta\leq k,$
 	  \begin{eqnarray*}
 	  	&& \partial^\beta q_j(x)=f(x_{\mu_{\ell_j}})^{-\abs{\beta}}\big(\partial ^{\beta} V\big)\big(x_{\mu_{\ell_j}}+f(x_{\mu_{\ell_j}})^{-1}x\big) \\
 	  	&=& f(x_{\mu_{\ell_j}})^{-\abs{\beta}}\sum_{0\leq\abs{\gamma}\leq k-\abs\beta}\frac{  f(x_{\mu_{\ell_j}})^{-\abs{\gamma}}\partial ^{\gamma+\beta} V(x_{\mu_{\ell_j}})}{\gamma !} x^{\gamma}\\
 	  	&&  +   f(x_{\mu_{\ell_j}})^{-(k+1)}  \sum_{\abs\gamma=k+1-\abs\beta}  \frac{ \abs\gamma x^{\gamma}}{\gamma !}\times  \\
&& \qquad\qquad\qquad \times\int_{0}^{1}  (1-\theta)^{\abs\gamma-1}\partial^{\gamma+\beta}  V\big(x_{\mu_{\ell_j}}+ \theta x  f(x_{\mu_{\ell_j}})^{-1}  \big) d\theta,
 	  \end{eqnarray*}
 	  with the remainder term above trending to $0$ as $j\rightarrow+\infty$ for any $x\in B_{r/\sqrt 2}.$   Meanwhile for the first term on the right hand side, we have 
 	  \begin{eqnarray*}
 	  &&	f(x_{\mu_{\ell_j}})^{-\abs{\beta}}\sum_{0\leq\abs{\gamma}\leq k-\abs\beta}\frac{  f(x_{\mu_{\ell_j}})^{-\abs{\gamma}}\partial ^{\gamma+\beta} V(x_{\mu_{\ell_j}})}{\gamma !} x^{\gamma}\\
 	  &=& \sum_{\stackrel{\alpha\geq\beta}{ \abs{\alpha}\leq k}}\frac{  f(x_{\mu_{\ell_j}})^{-\abs{\alpha}}\partial ^{\alpha} V(x_{\mu_{\ell_j}})}{\inner{\alpha-\beta} !} x^{\alpha-\beta} \longrightarrow  \sum_{\stackrel{\alpha\geq\beta}{ \abs{\alpha}\leq k}}\frac{ A_\alpha}{\inner{\alpha-\beta} !} x^{\alpha-\beta}=\partial^\beta p(x)
 	  \end{eqnarray*}
 as $j\rightarrow+\infty,$ the last line using \eqref{alapha+} and \eqref{defofp+}.  Combining the above relations we obtain the first assertion \eqref{conv}. 

It remains to show \eqref{conforp}.  Recall 
\begin{eqnarray*}
q_j(x)=V(x_{\mu_{\ell_j}}+f(x_{\mu_{\ell_j}})^{-1}x)-V(x_{\mu_{\ell_j}}).
\end{eqnarray*}
It then follows from \eqref{conv} that for any $1\leq \abs\beta\leq 2$ and  for any $x\in B_{r/\sqrt 2}$ we have
\begin{eqnarray*}
 \partial^\beta p(x)=\lim_{j\rightarrow+\infty}\partial^\beta q_j(x)=\lim_{j\rightarrow+\infty}f(x_{\mu_{\ell_j}})^{-\abs\beta}\inner{\partial^\beta V}\big(x_{\mu_{\ell_j}}+f(x_{\mu_{\ell_j}})^{-1}x\big).
\end{eqnarray*}
This implies for any $x\in  B_{r/\sqrt 2}$ we have, using the notation \eqref{mrho}
\begin{eqnarray}\label{highests}
	 \abs{\partial_x p(x)}=\lim_{j\rightarrow+\infty} f(x_{\mu_{\ell_j}})^{-1} 
	 \abs{(\partial_x V) \big(x_{\mu_{\ell_j}}+f(x_{\mu_{\ell_j}})^{-1}x\big)},
\end{eqnarray}
\begin{eqnarray}\label{eest}
	 \lambda_{p,i}(x)=\lim_{j\rightarrow+\infty} f(x_{\mu_{\ell_j}})^{-2} \lambda_{V,i}\big(x_{\mu_{\ell_j}}+f(x_{\mu_{\ell_j}})^{-1}x\big)
\end{eqnarray}
and
\begin{eqnarray}\label{eest+12017}
	\mathcal M_p(x)= \lim_{j\rightarrow+\infty}f(x_{\mu_{\ell_j}})^{-2}  \mathcal M_V \big(x_{\mu_{\ell_j}}+f(x_{\mu_{\ell_j}})^{-1}x\big).
\end{eqnarray}
Now let $x\in   B_{r/\sqrt 2} $ and let  $i\in I_{x,p}.$ Then we have $\lambda_{p,i}(x)>0$,  which along  with \eqref{eest}  yields  
\begin{eqnarray*}
	 \lambda_{V,i}\big(x_{\mu_{\ell_j}}+f(x_{\mu_{\ell_j}})^{-1}x\big)>0
\end{eqnarray*}
for all $j$ large enough.  As a result  it follows from \eqref{maiass}  that  
\begin{multline*}
	 \lambda_{V,i}\big(x_{\mu_{\ell_j}}+f(x_{\mu_{\ell_j}})^{-1}x\big)\\
	 \leq C_5\Big(\mathcal M_V\big(x_{\mu_{\ell_j}}+f(x_{\mu_{\ell_j}})^{-1}x\big)+\big|\partial_x V\big(x_{\mu_{\ell_j}}+f(x_{\mu_{\ell_j}})^{-1}x\big)\big|^2\Big)\\
	  +C_5\Big(\sum_{2\leq \abs\alpha \leq k}\abs{\partial^\alpha  V\big(x_{\mu_{\ell_j}}+f(x_{\mu_{\ell_j}})^{-1}x\big)}^{(2-\delta_1)/\abs\alpha}+1\Big),
\end{multline*} 
which holds for all $j$ large enough.  Then using \eqref{highests}-\eqref{eest+12017} yields, for any   $i\in I_{x,p}$ with $x\in   B_{r/\sqrt 2} ,$
\begin{eqnarray*}
	&&\lambda_{p,i}(x)=\lim_{j\rightarrow+\infty} f(x_{\mu_{\ell_j}})^{-2} \lambda_{V, i}\big(x_{\mu_{\ell_j}}+f(x_{\mu_{\ell_j}})^{-1}x\big)\\
	&\leq &  C_5 \lim_{j\rightarrow+\infty} \bigg[   f(x_{\mu_{\ell_j}})^{-2}  \mathcal M_V \big(x_{\mu_{\ell_j}}+f(x_{\mu_{\ell_j}})^{-1}x\big)\\
	&&+ \big|f(x_{\mu_{\ell_j}})^{-1}  \partial_x V\big(x_{\mu_{\ell_j}}+f(x_{\mu_{\ell_j}})^{-1}x\big)\big|^2\\
	&& +  f(x_{\mu_{\ell_j}})^{-\delta_1}\sum_{2\leq \abs\alpha \leq k}\abs{f(x_{\mu_{\ell_j}})^{-\abs\alpha}  \partial^\alpha  V\big(x_{\mu_{\ell_j}}+f(x_{\mu_{\ell_j}})^{-1}x\big)}^{\frac{2-\delta_1}{\abs\alpha}}\\
	&&+ f(x_{\mu_{\ell_j}})^{-2}\bigg]   \\
	&=&C_5\inner{\mathcal M_p(x)+\abs{\partial_x p(x)}
^2},
	\end{eqnarray*}
	the last line holding because $f(x_{\mu_{\ell_j}})^{-1}\rightarrow 0$ as $j\rightarrow+\infty$ and for any $ 2\leq \abs\alpha \leq k$ we have 
\begin{eqnarray*}
	 &&\abs{f(x_{\mu_{\ell_j}})^{-\abs\alpha}  \partial^\alpha  V\big(x_{\mu_{\ell_j}}+f(x_{\mu_{\ell_j}})^{-1}x\big)}\\
	 &\leq&  \inner{f(x_{\mu_{\ell_j}})^{-1}f\big(x_{\mu_{\ell_j}}+f(x_{\mu_{\ell_j}})^{-1}x\big)}^{\abs\alpha}\leq C_6
\end{eqnarray*}
due to \eqref{7310}.
 We have proven \eqref{conforp}.
 
 \bigskip
 \noindent
  \underline{{\it Step 3}}.  Let  $\zeta_j, j\geq 1,$ be given in Step 1.   Observe $\zeta_j\in C_0^\infty(B_{r/\sqrt 2})$  for all $j.$  
 Then in view of  the condition \eqref{uninor}, we   conclude that there exists a subsequence of $\zeta_j,$ still denoted by $\zeta_j,$  and a $\zeta\in H_0^1(B_{r/\sqrt 2})$ such that  
  \begin{eqnarray}\label{weakcon}
  \zeta_j  \rightarrow \zeta ~~~\textrm{weakly in }~H_0^1(B_{r/\sqrt 2}), 
  \end{eqnarray}
   and   
   \begin{eqnarray}\label{strcon}
 \zeta_j  \rightarrow \zeta ~~~\textrm{  strongly in }~L^2(B_{r/\sqrt 2}) 
  \end{eqnarray}
 due to the  compact injection of $H_0^1(B_{r/\sqrt 2})$ into $L^2(B_{r/\sqrt 2}).$   
 The weak convergence \eqref{weakcon} implies
 \begin{eqnarray*}
 \norm{  \zeta}_{L^2}^2+\norm{\partial_x \zeta}_{L^2}^2 &\leq &\inner{\liminf_{j\rightarrow+\infty} \norm{\zeta_j}_{H^1}}^2\leq 	 \liminf_{j\rightarrow+\infty} \norm{\zeta_j}_{H^1}^2\\
 & = &\liminf_{j\rightarrow+\infty} \norm{\partial_x\zeta_j}_{L^2}^2+ \norm{  \zeta}_{L^2}^2,
 \end{eqnarray*}
 the last equality using \eqref{strcon}. Thus
 \begin{eqnarray}\label{parzeta}
 \norm{\partial_x \zeta}_{L^2}^2 \leq \liminf_{j\rightarrow+\infty} \norm{\partial_x\zeta_j}_{L^2}^2.
 \end{eqnarray}
From \eqref{uninor},  \eqref{conv} and  \eqref{strcon} it follows that, observing  supp\,$\zeta_j\subset B_{r/\sqrt 2}$  for all $j\geq 1,$
\begin{eqnarray}\label{zetaj}
 \lim_{j\rightarrow+\infty}\int_{\mathbb R^n}\inner{ \tau^2 \abs{\partial_x q_j}^2- \tau \Delta q_j }\abs{\zeta_j}^2dx = \int_{\mathbb R^n}\inner{ \tau^2 \abs{\partial_x p}^2- \tau \Delta p }\abs{\zeta}^2dx,
\end{eqnarray}
and
\begin{eqnarray}\label{+52}
 \lim_{j\rightarrow+\infty}\norm{  \tilde f_{\tau q_j} \zeta_j}_{L^2}^2= \lim_{j\rightarrow+\infty}\norm{\tilde f_{\tau p}\,\zeta}_{L^2}^2.
\end{eqnarray}
Consequently, observe 
\begin{eqnarray*}
	 \norm{\big(\partial_{x}+ \tau\inner{\partial_{x} p}  \big)\zeta}_{L^2 }^2=\norm{\partial_x\zeta}_{L^2}^2+ \int_{\mathbb R^n}\inner{ \tau^2 \abs{\partial_x p}^2- \tau \Delta p }\abs{\zeta}^2\,dx,
\end{eqnarray*}
and thus   using \eqref{parzeta}-\eqref{+52}  gives
\begin{eqnarray}\label{es+170623}
\begin{split}
 	& \norm{\big(\partial_{x}+ \tau\inner{\partial_{x} p}  \big)\zeta}_{L^2 }^2- \eps  C_2 \norm{\tilde f_{\tau p}  \zeta}_{L^2}^2 \\
 	\leq &  \liminf_{j\rightarrow+\infty}\Big[ \inner{\triangle_{\tau q_j}^{(0)}\zeta_j ,  \zeta_j}_{L^2 } - \eps C_2 \norm{\tilde f_{\tau q_j}(x)  \zeta_j}_{L^2}^2\Big] 
 	 \leq   0,
 	 \end{split}
\end{eqnarray}
the last inequality following from \eqref{cruc1}.  
 Moreover in view of \eqref{conforp} we can apply  Corollary \ref{corcr} to conclude that, decreasing $r$ if necessary so that $r/\sqrt 2\leq \sigma$ with $\sigma$ given in Corollary \ref{corcr},  
 \begin{eqnarray}\label{leri}
		  \norm{\partial_{x}\zeta}^2_{L^2}
    +  \tau^2 \norm{ \inner{\partial_{x}p} \zeta}^2_{L^2}\leq
    C_7 \Big(\norm{\big(\partial_{x}+ \tau\inner{\partial_{x} p}  \big)\zeta}_{L^2 }^2+\norm{\zeta}_{L^2}^2\Big).
	\end{eqnarray}
	Here the constant $C_7$ may depend on the polynomial $p,$ but is independent of $\tau.$  
	On the other hand,  note $p\in E_k$ and then we can use the Baker-Campbell-Hausdorff formula (see \cite[Lemma 4.14]{Nier-BJ} for instance) to obtain  that 
 \begin{eqnarray*}
 	 \tau^{2/k}\norm{\zeta}_{L^2}^2 \leq  C_8 \norm{\tilde f_{\tau p}  \zeta}_{L^2}^2\leq C_9\Big(\norm{\partial_x u}_{L^2}^2+\tau^2\norm{  \inner{\partial_x p} \zeta}_{L^2}^2\Big).
 \end{eqnarray*}
 This along with \eqref{leri} implies that 
\begin{eqnarray}\label{17071201}
 	 \norm{\tilde f_{\tau_0 p}(x)  \zeta}_{L^2}^2\leq C_{10} \norm{\big(\partial_{x}+ \tau_0\inner{\partial_{x} p}  \big)\zeta}_{L^2 }^2
 \end{eqnarray}
 for some $\tau_0$ large enough.  Note 
  \eqref{es+170623} holds for arbitrary  $\tau$ and thus we combine the above estimate and   \eqref{es+170623}  to get 
  \begin{eqnarray*}
   \norm{\big(\partial_{x}+ \tau_0\inner{\partial_{x} p}  \big)\zeta}_{L^2 }^2  \leq  \eps \,C_2 \norm{\tilde f_{\tau_0 p}  \zeta}_{L^2}^2\leq  \eps \,C_2C_{10} \norm{\big(\partial_{x}+ \tau_0\inner{\partial_{x} p}  \big)\zeta}_{L^2 }^2.	
 \end{eqnarray*}
 Thus letting $\eps= 1/\inner{2C_2C_{10}}$ we obtain $ \norm{\big(\partial_{x}+ \tau_0\inner{\partial_{x} p}  \big)\zeta}_{L^2 }=0,$ and thus $\norm{\zeta}_{L^2}=0$ in view of \eqref{17071201}.  Furthermore using \eqref{leri} for $\tau=\tau_0$ gives $\norm{\partial_x\zeta}_{L^2}=0$. This contradicts \eqref{weakcon} and \eqref{strcon}, since $\norm{\zeta_j}_{H_0^1}=1$ by \eqref{uninor}.

 The contradiction yields the first property \eqref{cesc}  in Theorem \ref{thmmain}. 
 \end{proof}
 
 \begin{proof}[Completeness of  the proof of  Theorem \ref{thmmain}] In this part we will prove the second property \eqref{lastestimate} in Theorem \ref{thmmain}.    Recall we have already proven that 
 \begin{equation}\label{fies}
 		\forall~\tau\geq\tau_0,~~\forall~u\in C_0^\infty(\mathbb R^n),\quad  \norm{ \tilde f_\tau u}_{L^2}^2
     \leq C   \inner{\triangle_{\tau V}^{(0)} u, u}_{L^2} + C\norm{u}_{L^2}^2,
 \end{equation}
 for some $\tau_0>0$ and $C\geq 1.$      It remains to consider $\tau$ with $0<\tau<\tau_0.$   
 
 Let $\tau_0$ and $C$ be the constants in \eqref{fies}.     For any $\tau$ with $0<\tau<\tau_0$ we take $m=m_\tau$ by
 \begin{eqnarray*}
 	m=\max\Big\{1, ~\sqrt{(2C-1)/2C} \tau_0/\tau\Big\}.
 \end{eqnarray*}
 Then direct verification shows 
\begin{equation}
	\label{mvalue}
1-	\frac{1}{2C}\leq \big(m\tau / \tau_0\big)^2\leq 1.
\end{equation}
 Note $m\geq 1$ and thus  we have the comparison in the sense of quadratic forms on $C_0^\infty(\mathbb R^n)$: 
 \begin{eqnarray*}
 	 \triangle_{\tau_0 V}^{(0)}   &\leq& m  \triangle_{\tau V}^{(0)}  + (\tau_0^2-m\tau^2)\abs{\partial_x V}^2  - (\tau_0-m\tau) \Delta V\\
 	 &\leq&   m  \triangle_{\tau V}^{(0)} + \Big(1- \big(m\tau / \tau_0\big)^2\Big) \tilde f_{\tau_0} ^2\\
 	 &\leq&   m  \triangle_{\tau V}^{(0)} + \frac{1}{2} \triangle_{\tau_0 V}^{(0)} +\frac{1}{2},
 \end{eqnarray*}
 the last inequality holding because it follows from   \eqref{mvalue} and  \eqref{fies}  that,  for any $u\in C_0^\infty(\mathbb R^n),$
 \begin{eqnarray*}
   \Big(1- \big(m\tau / \tau_0\big)^2\Big) \norm{\tilde f_{\tau_0} u}_{L^2}^2\leq  \frac{1}{2C}\norm{\tilde f_{\tau_0} u}_{L^2}^2\leq  {1\over 2} \inner{\triangle_{\tau_0 V}^{(0)} u, u}_{L^2} +{1\over 2}\norm{u}_{L^2}^2.
 \end{eqnarray*}
 Consequently we have
 \begin{eqnarray*}
 	0\leq  \triangle_{\tau_0 V}^{(0)}   \leq  2 m  \triangle_{\tau V}^{(0)}  +1,
 \end{eqnarray*}
 which yields that, for any $u\in C_0^\infty(\mathbb R^n)$ and for any $0<\tau<\tau_0,$
 \begin{eqnarray*}
 	  \norm{ \tilde f_\tau u}_{L^2}^2\leq   \norm{ \tilde f_{\tau_0} u}_{L^2}^2
    & \leq &C   \inner{\triangle_{\tau_0 V}^{(0)} u, u}_{L^2} + C\norm{u}_{L^2}^2\\
    & \leq&  2mC   \inner{\triangle_{\tau V}^{(0)} u, u}_{L^2}+2C\norm{u}_{L^2}^2.
 \end{eqnarray*}
 This gives \eqref{lastestimate}, completing the proof of 
  Theorem \ref{thmmain}.
\end{proof}

\end{document}